\documentclass[a4paper,reqno, 11pt]{amsart}
\usepackage{graphicx,color,marginnote}
\usepackage{footmisc}
\usepackage{amsmath}
\usepackage[utf8]{inputenc}
\usepackage[T1]{fontenc}
\usepackage{cite}
\usepackage[english]{babel}

\usepackage{amsfonts}
\usepackage{amssymb}
\usepackage{bbm}
\usepackage{epstopdf}
\usepackage{color}
\numberwithin{equation}{section}
\usepackage[marginparwidth=3cm, marginparsep=.5cm]{geometry}

\newcommand{\cT}{\mathcal T}
\newcommand{\cH}{\mathcal H}
\newcommand{\cE}{\mathcal E}
\newcommand{\cV}{\mathcal V}
\newcommand{\cI}{\mathcal I}
\newcommand{\cS}{\mathcal S}

\newcommand{\bbR}{\mathbb{R}}
\newcommand{\bbZ}{\mathbb{Z}}
\newcommand{\bbC}{\mathbb{C}}

\newcommand{\bbP}{\mathbb{P}}

\renewcommand{\P}{\mathbb{P}}
\newcommand{\E}{\mathbb{E}}

\newcommand{\ve}{\varepsilon}
\renewcommand{\d}{ {\sharp \delta} }

\newcommand{\im}{i_{\min}}
\newcommand{\iM}{i_{\max}}

\usepackage{cleveref}
  \crefname{theorem}{Theorem}{Theorems}
  \crefname{lemma}{Lemma}{Lemmas}
  \crefname{remark}{Remark}{Remarks}
  \crefname{proposition}{Proposition}{Propositions}
\crefname{notation}{Notation}{Notations}
\crefname{claim}{Claim}{Claims}
  \crefname{definition}{Definition}{Definitions}
  \crefname{corollary}{Corollary}{Corollaries}
  \crefname{section}{Section}{Sections}
  \crefname{figure}{Figure}{Figures}
    \crefname{assumption}{Assumption}{Assumptions}

\usepackage{amsthm}
\newtheorem{theorem}{Theorem}

\newtheorem{remark}[theorem]{Remark}
\newtheorem{corollary}[theorem]{Corollary}

\newtheorem{assumption}[theorem]{Assumption}
\newtheorem{claim}[theorem]{Claim}
\newtheorem{lemma}[theorem]{Lemma}
\newtheorem{definition}[theorem]{Definition}
\newtheorem{proposition}[theorem]{Proposition}

\author{Benoît Laslier}

\title[Local limits are stable under bounded boundary perturbations]{Local limits of lozenge tilings are stable under bounded boundary height perturbations}

\date{}

\begin{document}

\begin{abstract}
We show that bounded changes to the boundary of a lozenge tilings do not affect the local behaviour inside the domain. As a consequence we prove the existence of a local limit in all domains with planar boundary. The proof does not rely on any exact solvability of the model beyond its links with uniform spanning trees.
\end{abstract}

\maketitle

\section{Introduction}


%
%

Random tilings of regions of the plane by lozenges have been studied in combinatorics and probability for a long time and are still an active domain of study. 
Indeed it is a reasonably simple and natural model for random surfaces embedded in $\bbR^3$ which has links with various fields of mathematics, see for example \cite{Tiliere2014} and references therein. 

In this paper we will be interested by the local behaviour of the model. More precisely, we will consider sequences of domains $D_n$, tiled by lozenges of fixed size, and we will try to understand the behaviour at this fixed local scale when the domains $D_n$ grow to infinity. Our purpose will be to show a form of decoupling result, meaning that the details of the boundary of $D_n$ (far away) do not affect the local behaviour around a fixed point.
There is however a subtlety when speaking of the details of the boundary. In this paper and in general for lozenge tilings, it is not very relevant to see the boundary as a curve in the plane; one has instead to think of a tiling as describing a surface and therefore of the boundary as a curve in $\bbZ^3$. This interpretation is the point of view we adopt when speaking of bounded perturbations.
\begin{theorem}\label{thm:intro}
Let $D_n$, $D'_n$ be two sequences of domains that exhaust the plane. Let $\gamma_n$ and $\gamma'_n$ be the corresponding boundary curves seen as embedded in $\bbZ^3$. Assume that the uniform lozenge tilings measure on $D_n$ converges locally around $0$ to a translation invariant, ergodic, non frozen, measure $\mu$. If there is a $C$ such that for all $n$, $d_\mathcal{H}(\gamma_n, \gamma'_n) \leq C$, then the uniform lozenge tilings measure on $D'_n$ also converges locally around $0$ to $\mu$.
\end{theorem}
See also \cref{thm:robustness} and \cref{prop:identification} for a more precise formulation and a more general result.

An interesting corollary of this result (together with the local limit of trees in T-graphs and the identification of T-graphs measures in \cref{prop:identification}) is a local limit for arbitrary planar boundary condition.
\begin{corollary}
Let $D_n$ be a growing sequence of domains and let $\gamma_n$ be the corresponding boundary curves. If there exists a plane $P$ such that $\gamma_n$ stays at bounded distance from $P$ and if the normal vector to $P$ has non-zero coordinates, then the uniform lozenge tilings measure of $D_n$ converges locally to the  (unique) translation invariant ergodic Gibbs measure associated with $P$.
\end{corollary}


To see the interest of such a result, it is useful to discuss the solvability of the model and its dependence on boundary conditions. Uniform lozenge tilings is an exactly solvable model, meaning that many quantities in the model have exact expressions.  Most notably, the number of tilings of a region of area $n$ is the determinant of an $n\times n$ adjacency matrix \cite{Kasteleyn1961}, and lozenge tilings can also be interpreted as interlacing particle arrays related to Schur processes.
 Actually, the existence of such formulas is one of the reasons the model is appealing since they can allow for a very fine understanding of its property at small and large scales. Furthermore some of the most surprising connections to other parts of mathematics appear through these exact formulas.

Yet maybe the main issue when studying dimers through solvability is that it is extremely sensible to boundary conditions, to the point that adding a single defect can break the analysis completely. To illustrate the kind of difficulty that can appear, in the related case of domino tilings more than 15 years were necessary to go from so called Temperley's boundary conditions \cite{Kenyon2007c} to Temperley's conditions with finitely many defects \cite{Russkikh2016}.
To the best of our knowledge, local limits for lozenge tilings have only been derived in two sets of cases \cite{Gorin2016,Petrov2014}, both including some perfectly straight lines on the boundary, with techniques related to the Schur representation of the problem.

 \medskip
 
The rest of the paper is organised as follows. In \cref{sec:previous}, we recall briefly the well known results about dimers and uniform spanning trees that are needed for the later sections, emphasising brevity over completeness as most reader should be familiar with the content. Section \ref{sec:robustness} contains the main argument of the paper, mostly showing that \cref{thm:intro} follows as soon as the limit measure satisfies a certain ``lack of concentration'' property. Finally in \cref{sec:spread}, we show that lozenge tilings measures indeed satisfy this property, which is the most involved part of the paper and where we  use the relation with uniform spanning trees.

\section{Previous results}\label{sec:previous}

In this section, we briefly recall some results on dimers and spanning trees needed for the main argument. For more details and precision, the reader can refer to \cite{Tiliere2014} for general facts on dimers, \cite{Kenyon1} for the generalised Temperley's bijection, and \cite{BLRannex} for the bijection with T-graphs.

\subsection{Dimer model}\label{sec:dimers}

Fix a weighted, undirected graph $G$, either finite or infinite. We call ``dimer configuration'' $M$ a subset of edges of $G$ such that each vertex is the endpoint of exactly one edge in $M$, or in other word a perfect matching between vertices of $G$ (two vertices being naturally matched if they are the endpoints of an edge in $M$). Edges in $M$ are often called dimers. 
In this paper, $G$ will always be a planar bipartite graph and we will also assume that it comes with a fixed, locally finite, embedding in the plane. 

Given a set of positive weights on edges $w(e)$, we will consider as usual the measure where each dimer carries a multiplicative weight $w$, i.e. if $G$ is a finite graph then we consider the measure 
\[
\mu_G(M) \propto \prod_{e \in M} w(e).
\]
For infinite graphs, we are interested in Gibbs measures defined with finite marginals as above.
The main application will be lozenge tilings with unit weights, i.e. the case where $G$ is the infinite hexagonal lattice and $w$ is identically equal to $1$, and the reader will not lose anything by thinking only of this case.
In that case the set of translation invariant Gibbs measures is well understood.
\begin{theorem}\cite{Sheffield2005}\label{thm:sheffield}
For every $p_a, p_b, p_c > 0$ with $p_a + p_b + p_c = 1$, there exists a unique Gibbs measure of the whole plane unit weight hexagonal lattice with the following properties. First it is translation invariant and ergodic, and second the densities of the three possible dimers orientations are respectively $p_a, p_b$ and $p_c$.
\end{theorem}


In our planar bipartite setting, it is also well known that a dimer configuration $M$ can be equivalently described by its \emph{height function} $h_M$, which is a function from the set of faces of $G$ to $\bbZ$ (or, depending on some convention, for any face $f$ there is a constant $h_0(f)$ independent of $M$ such that $h_M(f)$ takes value in $\bbZ + h_0(f)$ ).
This height function is almost Lipschitz : for any graph $G$, there exists a constant $C$ such that,
\begin{equation}\label{eq:lipschitz}
\forall f, f' \in G^\dagger, \forall M, \, \, |h_M(f) - h_M(f') | \leq C d(f, f') + C.
\end{equation}

%

\medskip

Let us elaborate a bit on our setting for taking local limits. Recall that we always work  with embedded graphs for ease of notation, so we can consider any dimer configuration (or later spanning tree) as a closed subset of $\bbC$. We define a local distance around $0$ between subsets as $d ( E, F) = \exp( - \max \{R : B(0, R) \cap E = B(0,R) \cap F \})$. Local convergence (around $0$) of a dimer measure means convergence in law according to this distance. We will also be using the following equivalent formulation : A sequence of measures $\mu_n$ converges locally to $\mu$ if and only if, for any $\ve>0$ and $R>0$, for any $n$ large enough, one can find a coupling between $\mu_n$ and $\mu$ represented by random variables $M_n$ and $M$ such that $\bbP\big( M_n \cap B(0, R)  =  M \cap B(0, R) \big) \geq 1 - \ve$.

To see why this is an equivalent formulation, observe that, in one direction, the Skorokhod embedding theorem actually constructs a coupling of all $M_n$ and $M$ under which the condition just becomes convergence in probability of the distance to $0$. For the other direction, note that finding such a coupling directly controls the Levy--Prokhorov metric so it implies convergence in law.

\begin{remark}
Our definition of local convergence depends on the embedding of the graph so that the distance between two embeddings of the same configuration can be large. This could easily be fixed but we restrain from doing so for the sake of brevity and because using this convention allows for easy notations.
\end{remark}

\subsection{Double dimer model}

We call ``double dimer model configurations'' the edge configurations obtained by taking the union of two dimer configurations.

If $G$ covers the whole plane, it is easy to see that a double dimer configuration is a disjoint union of single edges, simple loops and simple doubly infinite paths (where disjoint means they cannot share a vertex) spanning all vertices of $G$.
We will also consider the case where we superimpose two configurations of slightly different finite graphs $G$ and $G'$ (where $G$ and $G'$ are both sub-graph of a bigger one so that the superimposition makes sense). Then on $G \cap G'$, the double dimer configuration is an union of simple loops, single edges and simple paths connecting points in the symmetric difference $G \Delta G'$.


We will call ``double dimer model'' the law on double dimer configurations obtained by taking the union of independent samples of the dimer model. The key remark about this model is that if $M$ and $M'$ denote the independent matchings, then the law of $M$ and $M'$ given $M \cup M'$ is very simple.
Note that there are exactly two ways to obtain any given loop. Let us call the two bipartite classes white and black. When going clockwise around a loop we can see two patterns : either edges from white to black are in $M$ and edges from black to white are in $M'$ or the other way around. We will call these two possibilities the two \emph{orientations} of the loops. It is easy to see that to recover $M$ and $M'$ from $M\cup M'$, the orientations of all loops have to be sampled with probability $1/2$ independently for each loop.
Note that (when $G$ is infinite) there are also two possible orientations for infinite paths but their law is not necessarily straightforward as we cannot apply the Gibbs property to an infinite path. In the finite case, the orientation of a path is actually determined by its extremities so the reconstruction of $M$ and $M'$ from $M \cup M'$ is possible with only iid choices.

The height function also behaves well with respect to the double dimer confi\-guration. For any $M$ and $M'$, the function $h_M - h_{M'}$ (called double dimer height function) is obtained from the double dimer as follows. Say that one orientation is positive and the other negative. When crossing a loop from the inside to the outside, $h_M - h_{M'}$ increases by one if the loop is oriented positively and decreases by one if it is oriented negatively. For paths, the same convention extends in a straightforward way. When moving without crossing any loop or path, $h_M - h_M'$ stays constant. We will sometime refer to the paths and loops in the double dimer model as level lines of the double dimer height function.

\subsection{Uniform spanning tree}

For a (possibly oriented) graph $G$ with some marked boundary vertices $\partial$, a wired uniform spanning tree is a subset of oriented edges such that every vertex in $G\setminus \partial$ has a single outgoing edge and there is no cycle. Given weights on the oriented edges $w(\vec e )$, the uniform spanning tree measure is given by 
\[
\mu( \cT ) \propto \prod_{\vec e \in \cT} w(\vec e ).
\]

It is well known that sampling from that measure can be done through Wilson algorithm, sampling loop-erased random walks in arbitrary order. Let us elaborate a bit on the loop-erasure. Given $X = (X_i)_{i = 1, \ell}$ an arbitrary random path, we can define its forward loop erasure $\overset{\rightarrow}{E} (X)$ recursively as follows. If $X$ is a self avoiding path then $\overset{\rightarrow}{E}(X)= X$, otherwise define $t = \inf \{ i \leq \ell | \exists s < i, \, X_i = X_s \}$ and let $s$ be the (unique) time less than $t$ such that $X_s = X_t$. Then we define $\overset{\rightarrow}{E} (X)=  \overset{\rightarrow}{E} (X_1, \ldots, X_s, X_{t+1}, \ldots X_l$. We also define the backward loop erasure $\overset{\leftarrow}{E} X $ as the time reversal of $\overset{\rightarrow}{E} (X_\ell, \ldots, X_1)$. A well known result about loop-erased random walk is that, despite these two operations being different, the law of the result is the same.

\begin{lemma}\label{lem:back_loop_erasure}(\cite{Lawler2012} Lemma 7.2.1])
There exists a bijection $\phi$ from the set of finite random walk paths to itself such that for all $X$
\[
 \overset{\rightarrow}{E} X = \overset{\leftarrow}{E} \phi(X).
\] 
Furthermore $X$ and $\phi(X)$ use the same set of oriented edges with the same multiplicities.
\end{lemma}


In this paper we will be interested by graphs $G_n$ obtained as growing sub-graphs of a graph $G_\infty$ covering the plane. We will always assume that all our graphs come with a fixed embedding in the plane and we will always consider wired spanning tree where the wiring is along the single infinite face. The graph $G_\infty$ will be assumed to have the following \emph{uniform crossing} property as per \cite{BLR16}.





\begin{assumption}
  Let $R$ be the horizontal rectangle $[0,3]\times [0,1]$, $B_1 = B((1/2,1/2),1/4)$ be the 
``starting ball'' and $B_2= B((5/2,1/2),1/4) $ be the ``target ball'. There 
exists constants $n_0 > 0$ and $\alpha>0$ such that for all $z 
\in \bbC$, $n \geq n_0$, $v \in  n B_1$ such that $v+z \in G_\infty$,
\begin{equation}
  \P_{v+z}(X \text{ hits }n B_2+z \text{ before exiting } (nR+z)) 
>\alpha.\label{eq:cross_left_right}
\end{equation}
The same statement as above holds for crossing from right to left or in vertical rectangles.
\end{assumption}

Under the above assumption, it makes sense to speak of the wired uniform spanning tree of $G_\infty$, see Corollary 4.21 in \cite{BLR16}. An important fact for the next section is that this measure is supported on one-ended trees.

In \cref{sec:spread}, we will allow ourself to state results uniformly over finite graphs satisfying the uniform crossing property. By this we mean uniformly over graphs which can be extended into an infinite planar graph which satisfies the uniform crossing but we never use the actual extension.

\subsection{Dimers/Spanning tree correspondence}

The dimers and spanning tree measures are very closely related through two distinct relations.

The simplest one is the Temperley's bijection (or more precisely its generalisation from \cite{Kenyon1}). This is a bijection between a wired spanning tree on an arbitrary planar graph $G$ (with the wiring along an exterior face as usual) and the dimer model on a (planar bipartite) graph $G_D$ defined as follows : Consider $G$ with its embedding and the dual graph $G^\dagger$ ; the vertex set of $G_D$ is the union of $G \cup G^\dagger$ (forming one bipartite class), and of one vertex for each intersection of primal and dual edges (forming the other bipartite class). One then has to remove a single boundary vertex from $G_D$. Each ``intersection'' vertex is adjacent to the $4$ endpoints of the primal and dual edges. In the bijection, each dimer corresponds to a single oriented edge in either the tree or the dual and one can set up weights that are preserved by the bijection.

\medskip

The other one (introduced in \cite{Kenyon2007b}) is a bit more involved so we only describe it in the setting of uniform dimers on the infinite hexagonal lattice $\cH$.

For every $p_a, p_b, p_c>0$, one can construct a one parameter family of infinite planar graphs ``associated'' with these probabilities.
 The graphs constructed that way are called T-graphs.
 They are naturally constructed together with a planar embedding covering the plane and they satisfy the uniform crossing assumption (Theorem 3.8 in \cite{BLRannex}).
 Furthermore white vertices of the hexagonal lattice are in bijection with faces of $T$. 

The relation between dimers and trees is that there exists a function $\phi$ from the set of one-ended spanning trees of $T$ to dimers configurations on $\cH$ such that if $\cT$ is a wired uniform spanning tree of $T$ then $\phi( \cT) $ is a whole plane Gibbs measure on $\cH$.
To elaborate a bit more, note that the dual of a one-ended spanning tree is also a one ended spanning tree. The map $\phi$ first considers the dual spanning tree of $\cT$, with edges oriented to infinity, then each oriented dual edge is mapped in a local way to a single dimer.

\medskip

Let us emphasize here that for weighted graphs, the dual of a uniform spanning tree of a graph $G$ is \emph{not} a uniform spanning tree of the dual graph $G^\dagger$. For example it is easy to convince one-self that it is not in general possible to transport weights from the primal graph onto the dual. In particular it is not possible to sample directly the dual tree of a T-graph with Wilson's algorithm.

\begin{figure}\label{fig:orientation}
\begin{center}
\def\svgwidth{0.5\textwidth}
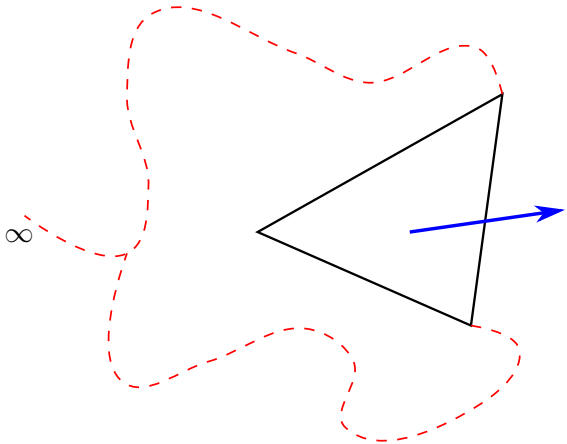
\caption{The configuration of the two primal paths (in dashed red) determines the orientation of the dual edge (blue arrow).}
\end{center}
\end{figure}

Also note that both mappings use the orientations of the dual edges, which is \emph{not} a local function of the primal tree. However note that, if $(xy)$ is a primal edge, the occupation and orientation of the dual edge $(xy)^\dagger$ is fully determined once we know the branches in the primal tree starting from $x$ and from $y$ (see \cref{fig:orientation}).

\medskip

Finally the height function behaves well in both correspondences. More precisely one can choose the conventions in the definition of the height function such that for all faces $x$, $y$ in the dimer graph, if $M$ is associated to a tree $\cT$ and $\gamma$ is the path in $\cT$ between $x$ and $y$,
\[
h_M(x) - h_M(y) = \frac{1}{2 \pi} W_{\text{int}}(\gamma),
\]
where $W_{\text{int}}$ is the intrinsic winding of a curve, i.e the change in $\arg \gamma'$ between the two endpoints. (To be completely accurate, one needs to do some local surgery to go from faces in the dimer graph to vertices in the tree graph. We do not dwell on these points because they only introduce deterministic $O(1)$ changes see \cite{Kenyon1,BLRannex} for full details.)
Finally the intrinsic winding above can also be computed using ``topological winding'' thanks to the following lemma. We define the topological winding of a curve $\gamma$ (assumed without loss of generality to be parametrised in $[0,1]$) around a point $z \notin \gamma$ as $W(\gamma, z) = \arg( \gamma(1) - z) - \arg(\gamma(0) -z)$, with the argument measured continuously along $\gamma$. If the endpoints are smooth, then $W(\gamma, \gamma(1) ) := \lim_{t \rightarrow 1} W(\gamma[0, t], \gamma(1))$ and $W(\gamma, \gamma(0) ) := \lim_{t \rightarrow 0} W(\gamma[t, 1], \gamma(0))$ are well defined.
\begin{lemma}\label{lem:intrinsic->top}[Lemma 2.1 in \cite{BLR16}]
Let $\gamma$ be a smooth self avoiding curve parametrised in $[0, 1]$with $\gamma'(s) \neq 0$ for all $s$. We have
\begin{equation}
W_{\text{int}}(\gamma) = W(\gamma, \gamma(1) ) + W( \gamma, \gamma(0) ).\label{eq:int->top_gen}
\end{equation}
\end{lemma}

\section{Robustness}\label{sec:robustness}

The key argument in this section is to play two observations against each other. On the one hand the effect of boundary conditions is ``carried'' only through level lines of the double dimer height function connecting to the boundary and is therefore only directly felt at height $O(1)$. On the other hand the height function at a point has diverging variance so the height will almost never be $O(1)$.

\subsection{The spread-out property}

In this section we give a formal definition of the second observation to be able to state and prove the main theorem. The proof that it actually holds for whole plane lozenge tiling measures is the most involved part of the paper and will be the purpose of \cref{sec:spread}.


\smallskip

Given $R > 0$ and a configuration $m$, we will write $\pi(R, m)$ for the dimer measure on configurations that agree with $m$ outside of $B(0, R)$. Note that with this notation, the Gibbs property for a measure $\mu$ states that $\mu( . | \, M|_{B(0, R)^c} ) = \pi(R, M)$.
\begin{definition}
We say that a dimer measure $\mu$ has the spread-out property if it satisfies the following. For all $\ve > 0$, there exists $R > 0$ such that,
\[
 \bbP_\mu\left[ \forall x \in \bbR, \, \bbP_{\pi(R, M)}[  h(0) \in (x, x+1)]   \leq \ve   \right] > 1 - \ve,
\]
where $h$ is defined by fixing the height of any point outside $B(0, R)$.
In other words, with hight probability the configuration outside of a large enough ball does not ``freeze" the height inside significantly.
\end{definition}
Note that the definition makes sense independently of the reference flow used because it gives a condition on all $x$.

\subsection{Proof of robustness}

In this section, we prove our main result that local convergence to any spread out Gibbs measure is robust under bounded perturbations of the boundary.

\begin{theorem}\label{thm:robustness}
Let $D_n$ be a sequence of (possibly random) domains containing $0$ and assume that the dimer measures on $D_n$ converge locally around $0$ to a spread-out infinite volume measure $\mu$. Let $D_n'$ be another sequence of (possibly random) domains containing $0$ and such that there exists $K > 0$ such that, for almost all realisation of $(D_n, D'_n)$, for any two sequences of configurations $M_n$ on $D_n$ and $M'_n$ on $D'_n$,
\[
\forall n, \sup_{v \in \partial (D_n \cap D_n')} |h_{M'}(v) - h_{M}(v)| \leq K .
\]
Then the uniform measures on $D'_n$ also converge locally around $0$ to $\mu$.
\end{theorem}

Note that if the domains are random, then $D_n$ and $D_n'$ have to be defined in the same probability space and that the coupling is important in the hypothesis even though it does not appear in the conclusion.

\begin{proof}

For ease of notation, we drop $n$ from the notations whenever it is not explicitly varying and denote by $h$ the double dimer height function, i.e. $h = h_{M'} - h_M$.

Recall from the introduction that, to any pair of dimer configurations, we can associate a family of loops, edges and paths connecting boundary points and that these paths and loops are level lines of $h$. In particular, from the assumption on the boundary height, we see that any face adjacent to a path has height bounded by $K+1$.

Let us define three coupled configurations $M$, $M'$ and $M''$ as follow. $M$ and $M'$ are independent configurations in $D$ and $D'$. $M''$ is a matching of $D'$ obtained by considering the double dimer configuration $M \cup M'$ and changing the orientation of all loops from $M'$. Formally an edge $e$ is in $M''$ in exactly these two cases : either $e \in M'$ and $e$ is not in a loop of $M \cup M'$, or $e$ is in a loop of $M \cup M'$ and $e \in M$.

It is easy to see that $M$ and $M''$ agree everywhere on $D \cap D'$ except on the paths defined by $M \cup M'$. Furthermore since the conditional laws of $M$ and $M'$ given  $M \cup M'$ are given by iid orientations of the loops, we clearly have :

\begin{claim}
$M''$ has the same law as $M'$.
\end{claim}

Recalling the characterisation of local convergence from the end of \cref{sec:dimers}, we see that we only need to prove that the paths in $M \cup M'$ stay away from $0$ with hight probability.


Now fix $\ve$ and $r$ and let $C$ be the constant from the Lipschitz property \eqref{eq:lipschitz}.  We also choose $R$ from the spread-out property suited for the constant $\ve$ and afterwards $n$ large enough so that a sample $M_\infty$ of $\mu$ can be coupled with $M_n$ on $\partial B(0, R)$ with probability at least $1-\ve$. 
Now let us explore the whole configuration $M_n'$, together with
 $M_n$ and $M_\infty$ everywhere outside $B(0,R)$. By union bound, with probability at least $1-2\ve$, we can assume that $M_n$ agrees with $M_\infty$ on $\partial B(0, R)$ and that 
\[
\forall x \in \bbR, \, \bbP_{\pi(R, M_n)}[  h(0) \in (x, x+1)  ] \leq \ve.
\]
In particular, conditioned on this event, $\bbP(|h(0)| \leq Cr + K +1 + C) \leq 2(Cr + K + C +1)\ve$ and overall $\bbP(|h(0)| \leq Cr + K +1 + C) \leq 2(Cr + K + C + 2)\ve$.

As we noted before, the height along a path is is at most $K+1$ so, recalling \eqref{eq:lipschitz}, we see that on the hight probability event $|h(0)| > Cr + K +1 + C$, no path intersects $B(0, r)$. We can also clearly extend $M_n$ and $M_\infty$ so that they agree on $B(0, R)$. Overall, we have a coupling of $M_n''$ and $M_\infty$ such that they agree on $B(0, r)$ with hight probability, which concludes the proof. 
\end{proof}

\begin{remark}
Note that we do \emph{not} need to consider the law of $M''$ at the step where we have discovered $M'$ and $M$ outside of $B(0,R)$. In particular the fact the $M''$ does not respect the domain Markov property once we condition on $M'$ is not an issue.
\end{remark}

\begin{corollary}\label{cor:translation_invariance}
If $\mu$ is an infinite volume Gibbs measure and $\mu$ satisfies the spread-out property, then $\mu$ is translation invariant.
\end{corollary}
\begin{proof}
Apply the previous result to the domains $D_n$ defined by conditioning $\mu$ outside of $B(0, n)$ and with $D'_n$ a translation of $D_n$ by $1$. Then on the one hand the law around $0$ in $D'_n$ is a translation of $\mu$ and on the other hand it converges to $\mu$.
\end{proof}

%

\section{Proof of the spread out property}\label{sec:spread}

In this section, we show that the translation invariant ergodic Gibbs measures on lozenge tilings satisfy the spread-out property, which concludes the proof of \cref{thm:intro}. The main difficulty is that the conditioning on everything outside a domain is not something easily accessible with standard techniques so we need to go through spanning tree coupling arguments. 

Let us also remark on a difficulty with the overall structure of the proof. It is not a priory clear that, for a given T-graph $T$, the associated whole plane Gibbs measure $\phi(\cT)$ is translation invariant. Conversely, it is also not clear that one can represent the translations invariant Gibbs measure as coming from a tree. We will actually make this connection here, proving first that measures of the form $\phi(\cT)$ satisfy the spread out property and then concluding by \cref{cor:translation_invariance} and \cref{thm:sheffield}.

In this section, to be closer to the setup of similar proofs in \cite{BLR16}, we change our scaling convention and we think of graphs $D^\d$ with a mesh size $\delta$ and a fixed diameter instead of growing domains $D_n$ with a fixed scale. The main difference is with uniform crossing property that has to be restated according to this new convention, see for example \cite{BLR16} Section 4.1. When speaking of the wired UST, we always consider the wired boundary to be all vertices adjacent to the unbounded face of $D^\d$. With a slight abuse of notation, we sometime identify a graph with the union of its closed bounded faces.

\subsection{Decoupling in tree measures}

In this section, we will show that, at the level of the spanning tree, the configuration outside a large ball becomes independent of the configuration around $0$.

\begin{lemma}\label{lem:TVdecoupling}
For all $\ve > 0$, there exists $R$ (depending only on $\ve$ and the uniform crossing constants) such that uniformly over $\delta \leq 1$, uniformly over all graphs $D^\d$ such that $B(0,R) \subset D^\d$  (and in particular for graphs covering the plane) and satisfying the uniform crossing assumption, the following holds. Let $\cT$ be the the wired UST in $D^\d$ and let $\cT_R$ be the sub-tree made of all the branches starting from a vertex in $B(0, R)^c$, or in other word the minimal sub-tree of $\cT$ spanning all vertices of $B(0,R)^c$. Then with probability at least $1-\ve$,
\[
 \cT_R \cap B(0,1) = \emptyset.
\]
\end{lemma}
\begin{proof}
First note that by playing with the scaling, the lemma is equivalent to saying that, for any $\delta$ small enough, uniformly over graphs with $B(0,1) \subset D^\d$, for $R$ large enough, with probability at least $1-\ve$,
\[
\cT_1 \cap B(0,1/R) = \emptyset.
\]
Fix $\ve > 0$, by the version of Schramm's finiteness theorem in Lemma 4.16 of \cite{BLR16}, one can sample all branches starting in the annulus $B(0,1) \setminus B(0,1/2)$ in such a way that, after the $K$ first branches, the probability that any of the others have diameter more than $1/8$ is at most $\ve$. By proposition 4.8 in \cite{BLR16} (establishing a tail estimate for the probability that loop-erased random walk comes close to a point), by chosing $R$ large enough we can ensure that the probability that any of the $K$ first branches intersects $B(0, 1/R)$ is at most $\ve$. This concludes the proof because after sampling all branches starting from $B(0,1) \setminus B(0,1/2)$, the ball $B(0, 1/R)$ is deterministically separated from $B(0, 1)^c$.
\end{proof}


\begin{lemma}\label{lem:decoupling}
Fix $\ve> 0$. There exists $R > 0$ such that, uniformly over $\delta \leq 1$, there exists a set $\cE$ of configurations on $\partial B(0,R)$ with the following properties. Firstly, it has probability at least $1- \ve$ under the measure $\phi(\cT)$. Secondly, for any $M \in \cE$, one can sample $\pi_M$ in two steps as follows. The first step is the sampling of a sub-tree with an a priori complicated law, but such that, with probability $1-\ve$, no branch intersects $B(0,1)$ and deterministically they separate $0$ from infinity. The second step is the sampling of a uniform spanning tree in the remaining domain which has wired boundary conditions.
\end{lemma}
\begin{proof}
Fix $\ve >0 $ and consider the associated $R$ as per \cref{lem:TVdecoupling}. As above we denote by $\cT_R$ the law of the sub-tree spanned by $B(0,R)^c$.  
We know that $\P[\cT_R \cap B(0,1) \neq \emptyset] \leq \ve$. We also note that deterministically the domain left by $\cT_R$ has Dirichlet boundary condition and, by Wilson's algorithm, adding to $\cT_R$ an independent wired spanning tree of $G \setminus \cT_R$ gives a fair sample of $\cT$.


Now let us fix a dimer configuration on $B(0, R)^c$ which with some abuse of notation we still call $M$.
Of course, one way to sample $\pi_M$ fairly is to first sample $\cT_R$ given $\pi_M$ and then to complete with the law of $\cT$ given $\cT_R$ and $M$. However since $M$ is a function of $\cT_R$, the conditional law of $\cT$ given $\cT_R$ and $M$ is the same as the law of $\cT$ given only $\cT_R$ which is identified above as a independent wired UST.


To complete the proof of the lemma, we only need to find the good set $\cE$. 
We have
$
  \P[\cT_R \cap B(0,1) \neq \emptyset] = \E[ \pi_M[\cT_R \cap B(0,1) \neq \emptyset]] \leq \ve
$ but $\pi_M[\cT_R \cap B(0,1) \neq \emptyset]$ is in $[0,1]$ so applying Markov's inequality we find
$$
\P [  \pi_M[T_R \cap B(0,1) \neq \emptyset] \geq \sqrt{\ve}]\leq \sqrt{\ve}$$ which concludes.
\end{proof}

To finish the proof, we will argue that, conditionally on everything else, the last step provides the spread out property.
\begin{proposition}\label{prop:non-concentration}
Fix $\ve > 0$, there exists a $\delta_0$ (depending only on $\ve$ and the constants in the uniform crossing) such that for all $\delta \leq \delta_0$, for any graph $D^\d$ such that $B(0,1) \subset D^\d$ and satisfying the uniform crossing property, the wired uniform spanning tree in $D^\d$ satisfies
\[
\forall x \in \bbR, \, \bbP[ h(0) \in (x, x+1) ] \leq \ve,
\] 
where the choice of constant in the height function is such that the boundary height is fixed.
\end{proposition}
The proof of this proposition will be postponed until \cref{sec:single-branch}. Before that we will give the applications to lozenge tiling and conclude the proof of \cref{thm:intro}, admitting \cref{prop:non-concentration}.

\subsection{Application to lozenge tilings}

In this section, we show that the whole plane translation invariant Gibbs measures on lozenge tilings are spread-out, showing that \cref{thm:robustness} implies \cref{thm:intro}.


Recall from \cite{BLRannex} that to a one-ended spanning tree $\cT$ of $G$, one can associate a dimer configuration $\phi(\cT)$. Furthermore, if $\cT$ is taken according to the infinite uniform spanning tree measure (as obtained by taking a local limit of wired spanning tree measures, see Corollary 4.20 in \cite{BLR16}) then $\phi(\cT)$ is a whole plane Gibbs measure. 

\begin{proposition}\label{prop:identification}
Let $G$ be a whole plane T-graph associated with $p_a, p_b, p_c > 0$ and let $\cT$ be a wired spanning tree of $G$, then the dimer measure $\phi(\cT)$ satisfies the spread-out property and is the unique translation invariant ergodic Gibbs measure with dimer densities $p_a, p_b, p_c$.
\end{proposition}

\begin{proof}
 We start by proving the spread-out property for the law of $\phi(\cT)$, which we call $\mu$. Fix $\ve > 0$ and let $\delta_0$ be chosen according to \cref{prop:non-concentration}. We also choose $R$ from \cref{lem:decoupling}, $\delta \leq \delta_0$ so that \cref{lem:decoupling} applies, and we let $\cE$ be the corresponding set of configuration. 
 
For any $M \in \cE$, we see that, with probability $1-\ve$, no branch intersect $B(0,1)$. Also, under the event that no branch intersect $B(0,1)$, we can apply \cref{prop:non-concentration} so overall, by union bound, we see that
\[
 \forall M \in \cE, \, \forall x, \, \bbP_{\pi_M} [ h(0) \in (x, x+1) ] \leq 2 \ve,
\]
and by construction
\[
 \bbP (\cE) \geq 1- \ve,
\]
which proves the spread out property.

Applying \cref{cor:translation_invariance}, we see that $\mu$ is translation invariant and is therefore a linear combination of the ergodic translation invariant Gibbs measures. The exact same proof as for Theorem 4.10 in \cite{BLRannex} shows that the ergodic decomposition has to be trivial.
\end{proof}

\begin{remark}
Comparing with Theorem 4.10 in \cite{BLRannex}, this is a strict enhancement because here we do not need to average over the parameter $\lambda$. Indeed in \cite{BLRannex} this averaging was needed to restore translation invariance while here it comes from the spread-out property and the robustness. Also, note that T-graphs are not translation invariant so the invariance is recovered only when going from the tree to the dimers.
\end{remark}

\subsection{Non-concentration for a single branch}\label{sec:single-branch}

From the previous section, we reduced the conditioned measure to a wired spanning tree measure on a random large domain. We will actually be able to prove that, uniformly over that domain, the height at $0$ is not concentrated.

The idea of the proof will be to decompose the walk at each scale similarly to \cite{BLR16} and to show that each scale can have an independent additive contribution to the winding. We will then conclude by the local CLT. Slightly more precisely, we will focus our attention on scales where it is easy to control that only finitely many pieces of random walk contribute to the winding, which we will call isolated scales.

\medskip

We start by setting up the definitions and notations for the scale decomposition. Let $r_i = e^{i}$ for $\im \leq i \leq \iM - 1$, where $\im$ will be specified later as the maximum index such that the RSW assumptions hold ($\im = \log \delta - C$ for an universal constant $C$) and $i_{\max}-1$ is the largest index such that $B(0,r_i) \subset D^\d$. We let $C_i$ be the circle of radius $r_i$ centred at $0$ for $\im  +1 \leq C_i < C_{\iM} -1$.

We inductively define a sequence of times $\tau_k$ and indexes $i(k)$ as follows. We let $\tau_0 = 0$ and by convention $i(0) =\im - 1$, then $\tau_1$ is the first crossing time of $C_{\im}$ and $i(1) = \im$. Having defined $\tau_k$ and $i(k)$, we let $\tau_{k+1}$ be the first crossing time of either $C_{i(k)-1}$ or $C_{i(k) +1}$ after $\tau_k$ and we let $i(k)$ be the index of the circle it crossed. Note that if $i(k) = \im$ we instead consider only the next crossing of $C_{\im +1}$. We also stop the random walk when it exits $D^\d$, calling this time $\tau_{k_{\max}}$ and setting by convention $i(k_{\max}) = i_{\max}$. When we speak of a \emph{crossing} without further specification, it will always refer to a time of the form $\tau_k$. Note also that by construction, $X_{\tau_k}$ is a point in $C_{i(k)}$ up to a discretisation error of size $O(\delta)$. We define $S := (X_{\tau_k})_{k\geq 0}$ the sequence of crossing positions. For any $j$ we also let $\cV_j$ be the sequence of crossings of the circle $C_j$, i.e $\cV_j = \{ k : i(k) = j \}$.

In this section, for a path $X_n$, we will write $X[a,b]$ 
for the union of all edges used between times $a$ and $b$ seen as closed sets in $\bbC$
or sometime with an abuse of notation $X[a,b]$ for the subpath between times $a$ and $b$.
 We will call a set of the form $X[\tau_k, \tau_{k+1}]$ an \emph{elementary piece of random walk}.

We want to condition on the sequence of crossing positions $\mathcal{S} := (X_{\tau_k})_{k \geq 0}$. For any $i_{\min} \leq i  < i_{\max}$, we set $\kappa_i = \max \{ k : i(k) = i \}$ and we say that scale $i$ is \emph{pre-isolated} if $\kappa_{i-1} = \kappa_i - 1 = \kappa_{i+1} -2$. In other word, a scale $i$ is pre-isolated if just before the last crossing of $C_i$ the walk was inside and, immediately after, the walk goes to $C_{i+2}$ and never comes back further than $C_{i + 1}$.

\begin{lemma}
Let $\cI$ be the set of pre-isolated scales which are also even. There exists constants $c, c' >0$ independent of $\delta$ such that
\[
\P \left(| \cI | \leq c(i_{\max} - i_{\min}) \right) \leq e^{- c'(i_{\max} - i_{\min}) } .
\]
\end{lemma}
\begin{proof}
This is essentially a special case of the proof of Lemma 4.12 in \cite{BLR16}. The idea is to discover the set $\cS$ is steps ``from the outside'' to be able to control its law. More precisely, for $i_{\min} \leq \ell < i_{\max}$, we let $\cS_\ell$ be defined similarly to $\cS$ but only considering the circles $C_\ell, \ldots, C_{i_{\max} -1}$. In other word, we define times $\tau_k^\ell$ and indices $i^\ell(k)$ as above but except that when $i^\ell(k) = \ell$, we set $\tau_{k+1}^\ell$ to be the first exit of $B(0, r_{\ell + 1})$ instead of the exit from the annulus. Note that $\cS^{i_{\max}-1} \subset  \cS^{i_{\max}-2} \subset \ldots \subset \cS^{i_{\min}} = \cS$.

We will condition successively of all $\cS^\ell$ with decreasing $\ell$ and show that, along that procedure, each scale has an independent positive probability to be pre-isolated. Note that conditionally on $\cS^\ell$, $X$ is given by independent pieces of random walks. Furthermore is is easy to see that $(X_{\tau_k})_{k \in \cV_i}$ is measurable with respect to $\cS^{i-1}$, therefore to see whether a scale $i$ is pre-isolated, it is enough to know $\cS^{i - 1}$.

Now fix $i$ even and condition on $\cS^{i+1}$. By Lemma 4.6 in \cite{BLR16}, the last piece of random walk starting from $C_{i+1}$ has a positive probability to hit $C_i$ before $C_{i+2}$. If this happens, then in $\cS^{i}$ the last visit to $C_{i+1}$ happens immediately after the last visit to $C_i$. Iterating this argument for $i-1$ and $i$, we see that in $\cS^{i-1}$, with positive probability, we have $\kappa^\ell_{i-1} = \kappa^\ell_i - 1 = \kappa^\ell_{i+1} -2$. As we noted above, this implies that $i$ is pre-isolated. 

We just showed that when we discover successively the $\cS^\ell$, pre-isolated scales can appear with positive probability conditionally independently at each step. The lemma then follows from an elementary concentration bound for iid variables.
\end{proof}

Now let us condition on the set $\cS$ and assumes that $| \cI | > c(i_{\max} - i_{\min})$.
Let us also condition on all the pieces of random walk corresponding to odd scales. We say that a pre-isolated scale $i$ is isolated if the piece of random walk $X[\tau_{\kappa_i -1}, \tau_{\kappa_i}]$ contains a cycle separating $0$ from infinity and if $X[\tau_{\kappa_i +1}, \tau_{\kappa_i+2}]$ does not intersects $B(0, e^{i + 6/7})$ (recall that by assumption, $X_{\tau_{\kappa_i+2}}$ in in $C_{i+2}$ so this is possible). By Lemma 4.4 and 4.5 in \cite{BLR16}, each pre-isolated scale has a positive probability to be isolated and in particular we have.

\begin{lemma}
let $\cI'$ be the set of isolated scales. There exists constants $c, c' >0$ independent of $\delta$ such that
\[
\P \left(| \cI' | \leq c(i_{\max} - i_{\min}) \right) \leq e^{- c'(i_{\max} - i_{\min}) } .
\]
\end{lemma}

Now we condition additionally on all the pieces of random walk except those of the type $X[\tau_{\kappa_i}, \tau_{\kappa_i} +1 ]$ for $i \in I'$. 
The idea will be that, thanks to the assumptions on what happens just before and after each of the $\kappa_i$, the erasures of loops will happen independently for all $i \in \cI'$. Further the contribution of each scale $i \in \cI'$ is ``truly random'' so overall by the local CLT the sum will be spread out.

\begin{figure}\label{fig:paths}
\begin{center}
\def\svgwidth{0.4\textwidth}
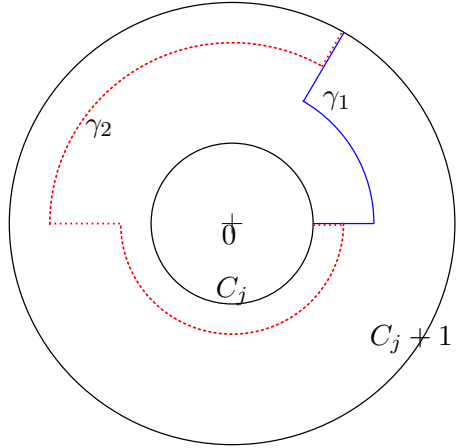
\caption{A schematic representations of the paths $\gamma_1$ (blue) and $\gamma_2$ (dotted red).}
\end{center}
\end{figure}

To formalise this idea and to go to a classical setting of an iid sum, we focus on a single scale $j$ and we introduce a coupling of two copies of the piece of random walk, which we call $X^1$ and $X^2$. The aim will be to obtain a coupling between the two loop-erasures such that they differ by one turn inside scale $j$ but agree everywhere outside. Note that we change the convention for the name of a scale to avoid conflict with the complex $i$. Without loss of generality, by rotating and symmetrising the lattice, we can assume that $\arg(X_{\tau_{\kappa_j}}) = 0$ and $\theta := \arg(X_{\tau_{\kappa_j + 1}}) \in [0, \pi]$. In this proof, we say that a walk $X$ \emph{follows} a curve $\gamma$ if their $L^\infty$ norm up to parametrisation is at most $e^j/12$. We define $\gamma^1$  and $\gamma^2$ as
\[
\gamma^1(t) =
\begin{cases}
 e^{j + t/2} &\text{for $t \in[0,1]$}\\
 e^{j + 1/2 + i\theta (t -1)} &\text{for $t \in[1,2]$}\\
e^{j + i \theta + (t-1)/2} &\text{for $t \in[2,3]$}\\
\end{cases}
,\quad
\gamma^2(t) =
\begin{cases}
 e^{j + t/3} & \text{for $t \in[0,1]$} \\
 e^{j + 1/3 - i \pi (t -1)} &\text{for $t \in[1,2]$} \\
e^{j + (t-1)/3 - i \pi } &\text{for $t \in[2,3]$} \\
e^{j + 2/3 -i\pi-i(\pi-\theta)(t-3)  } &\text{for $t \in[3,4]$} \\
e^{j + (t -2)/3 +i\theta  } &\text{for $t \in[4,5]$} \\
\end{cases}
\]

\begin{lemma}
Let $X$ be a piece of random walk, there exists a constant $c$ depending only on the constant in the uniform crossing such that
$
 \P(X \text{ follows } \gamma^1) > c
$ and $
 \P(X \text{ follows } \gamma^2) > c
$.
Further, conditioned on the first hitting of the circle of radius $e^{j + 1/6}$ and on following $\gamma^1$ or $\gamma^2$, the total variation distance between the laws of the first hitting point of the circle of radius $e^{j + 2/3}$ is upper bounded by $1 - c$.
\end{lemma}
\begin{proof}
The bounds on the probability to follow a path clearly follows from finitely many applications of Lemma 4.4 in \cite{BLR16}, with an application of equation 4.4 in \cite{BLR16} to show that, once the path has arrived close enough to the boundary, it can exit without backtracking.

For the bound on the law of the exit point, let $\tau_A$ be the exit time of $B(0, e^{j + 1/6})$ and $\tau_B$ be the first exit time of $B(0, e^{j + 2/3})$ and let also $F^1$ and $F^2$ be respectively the events that $X$ follows $\gamma_1$ and $\gamma_2$. Let $x$ be any point at distance at most $e^{j}/24$ of $e^{j + 2/3 + i \theta}$ and such that $\bbP[X_{\tau_B} = x] > 0$. Also let us condition on $X_{\tau_A}$ once and for all without further mention of it since it does not play any role. 

Note that we have
\begin{multline*}
	\bbP[X_{\tau_B} = x | F^1, X_{\tau_{\kappa_j + 1}} = y] \\
	= \frac{ \bbP(F^1 | X_{\tau_B} = x, X_{\tau_{\kappa_j + 1}}=y )  \bbP(X_{\tau_{\kappa_j + 1}}=y | X_ {\tau_B}=x) \bbP(X_{\tau_B} = x) }{ \bbP(F^1 | X_{\tau_{\kappa_j + 1}}=y) \bbP(X_{\tau_{\kappa_j + 1}}=y) }
\end{multline*}
and similarly for $F^1$, therefore
\[
\frac{\bbP[X_{\tau_B} = x | F^1, X_{\tau_{\kappa_j + 1}} = y]}{\bbP[X_{\tau_B} = x | F^2, X_{\tau_{\kappa_j + 1}} = y]} = \frac{\bbP(F^1 | X_{\tau_B} = x, X_{\tau_{\kappa_j + 1}}=y )   \bbP(F^2 | X_{\tau_{\kappa_j + 1}}=y)}{ \bbP(F^2 | X_{\tau_B} = x, X_{\tau_{\kappa_j + 1}}=y ) \bbP(F^1 | X_{\tau_{\kappa_j + 1}}=y). }
\]
We already argued that $\bbP(F^{1/2} | X_{\tau_{\kappa_j + 1}}=y)$ is lower bounded independently of $\delta$. For the other one, note that, under the conditioning, the piece of random walk is made of two independent random walks conditioned on their exit point from an annulus so $\bbP(F^{1/2} | X_{\tau_{\kappa_j + 1}}=y, X_{\tau_B} = x)$ is also lower bounded independently of $\delta$ for the same reason as above. Finally it is easy to see by Lemma 4.4 in \cite{BLR16} that $\bbP(X_{\tau_B} \in B(e^{j + 2/3 +i \theta}, e^j/24))$ is also lower bounded independently of $\delta$, which concludes the proof. 
\end{proof}

Thanks to the the lemma, we can construct a coupling of two pieces of random walk $X^1$ and $X^2$ such that the following event $E$ has positive probability. With an abuse of notation, it will be convenient to think of $X^1$ and $X^2$ as full trajectories of a random walk that agree everywhere except at scale $j$. Let $\tau_A$ be the first hitting time of the circle of radius $e^{j + 1/6}$ and $\tau_B$ be the first hitting time of the circle of radius $e^{j + 4/6}$. We set 
\begin{multline*}
E = \{X^1[\tau_{\kappa_j}, \tau_A ] = X^2[\tau_{\kappa_j}, \tau_A]\} \\
\cap \{ X^1[\tau_B, \tau_{\kappa_j +1}] = X^1[\tau_B, \tau_{\kappa_j +1}] \} \\
\cap \{ X^1 \text{ follows } \gamma^1 \} \cap \{ X^2 \text{ follows } \gamma^2 \}. 
\end{multline*}

Here it would be natural to just consider the standard forward loop erasures of $X^1[0,\tau_{\kappa_j +1}] $ and $X^2[0,\tau_{\kappa_j +1}] $ but this doesn't work since we cannot control that on the event $E$ these loop-erasures agree near $C_{j+1}$. To go around this issue, we use the fact that the law of the forward loop-erasure and the law of the backward loop-erasure are the same and we actually erase loops in a ``mixed'' way. 

\begin{lemma}\label{lem:LE}
Let $X$ be a random walk on a arbitrary graph, possibly conditioned on its end point. Let $T$ be a stopping time for the walk and $T_{\max}$ be its final time. Let $Y$ be defined from $X$ by the following loop-erasure procedure.
\begin{itemize}
\item First consider $X[0, T]$ and let $Y^T$ denote its forward loop-erasure. Parametrise $Y^T$ from $X_0$ to $X_T$ as $Y^T_s$,
\item then find $S = \min \{ s \geq 0 : Y^T_s \in X[T, T_{\max}] \}$ and $\tau = \max \{ t \geq 0 : X_t = Y_S^T \}$,
\item finally add to $Y^T[0, S]$ the backward loop-erasure of $X[\tau, T_{\max}]$.
\end{itemize}
Then $Y$ has the same law as the forward loop-erasure of $X$.
\end{lemma}
\begin{proof}
First note that if we were adding at the end the forward loop-erasure of $X[\tau, T_{\max}]$ instead of its backward loop-erasure, we would obtain exactly the forward erasure of the whole path $X$. It is therefore enough to show that the forward and backward loop-erasures of $X[\tau, T_{\max}]$ have the same law.

Now remark that $X[\tau, T_{\max}]$ is a simple random walk on a graph with a boundary (made of $Y^T[0, S]$ and the initial boundary) conditioned on its end-point. Indeed since $T$ is a stopping time $X[T, T_{\max}]$ is a simple random walk so $X[\tau, T_{\max}]$ is the law of a random walk after its last intersection with a fixed path. From \cref{lem:back_loop_erasure}, the law of the forward and backward loop-erasures of such a path are equal.
\end{proof}

Back to our problem, we define $Y^1(j+1)$ and $Y^2(j+1)$ by erasing loops as above from $X^1$ and $X^2$, using $\tau_A$ as the stopping time. Note that this can be seen as a coupling between copies of the law of the forward loop-erasures of $ X[0, \tau_{\kappa_j + 1}]$. 

\begin{lemma}
Recall that $j$ is assumed to be an isolated scale and assume that $X^1$ and $X^2$ agree everywhere outside of the piece at scale $j$. When the event $E$ occurs, the paths $Y^1(j+1)$ and $Y^2(j+1)$ satisfy, 
$$W(Y^2(j+1), 0) = W(Y^1(j+1), 0) + 2\pi,$$
and furthermore $Y^1$ and $Y^2$ agree everywhere after their first exit of $B(0, e^{j + 5/6})$.
\end{lemma}
\begin{proof}
Let us introduce $\tau_A^1, \tau_B^1$ and $\tau_A^2, \tau_B^2$ as respectively the first hitting times of the circles of radii $e^{j + 1/6}$ and $e^{j + 2/3}$ for $X^1$ and $X^2$.

Since $X^1$ and $X^2$ agree until $\tau^1_A = \tau^2_A$, their loop erasures agree so let us call it $Y^A$. Since $X^1[\tau_{\kappa_j -1}, \tau_{\kappa_j}]$ contains a loop inside $C_j$, and by construction of $\tau_A$, it is easy to see that $Y^A$ is a simple path in $B(0, e^{j + 1/6})$ from the center to the boundary. 

Let us introduce $S^1\min \{ s \geq 0 : Y^A_s \in X[\tau_A^1, \tau_{\kappa_j + 1}] \}$, $\tau^1_S = \max \{ t \geq 0 : X_t = Y_S^A \}$ and $s^1_{\max}$ the maximal index of $Y^1$. We also define $S^2$, $\tau^2_S$ and $s^2_{\max}$ similarly. By \cref{lem:LE}, we see that $Y^1[S^1, s^1_{\max}]$ is the backward loop erasure of $X[\tau_S^1,\tau_{\kappa_j + 1}]$, and similarly for $Y^2$.

Since $X^1$ follows $\gamma_1$, it is easy to see that $\tau_S^1 \leq \tau_B^1 \leq \tau_{\kappa_j + 1}$ so we can decompose the loop-erasure around time $\tau_B$. We write $Y^B$ for the backward loop erasure of $X^1[\tau_B^1, \tau_{\kappa_j + 1}]$ which by assumption is equal to the backward loop erasure of $X^2[\tau_B^2, \tau_{\kappa_j + 1}]$. It is easy to see that $Y^B$ does not return to $B(0, e^{j + 2/3})$ after its first exit of $B(0, e^{j + 5/6})$. Since $X^1[\tau^1_S, \tau^1_B]$ is a path in $B(0, e^{j + 2/3})$, the loops created by adding $X^1[\tau^1_S, \tau^1_B]$ to $Y^B$ cannot erase any part of $Y^B$ after its first exit of $B(0, e^{j + 5/6})^c$ and in particular $Y^1$ agrees with $Y^B$ after their first exit of $B(0, e^{j + 5/6})^c$. The same is true for $X^2$ which concludes the first part of the proof.

%
%
%
%
%
%
%
%
%

For the second part, first note that $Y^1$ and $Y^2$ are path with the same initial and ending points so their winding can only differ by a multiple of $2\pi$. Also each of them is obtained by erasing loops from $Y^A \cup X[\tau_A, \tau_{\kappa_j +1} ]$. It is easy to see that all the erased loops will be in $A(0, e^j, e^{j+1})$ and that $\Big(Y^A \cup X^{1/2}[\tau_A, \tau_{\kappa_j +1} ] \Big) \cap A(0, e^j, e^{j+1}) $ is a set at Hausdorff distance less than $e^{j}/12$ of $\gamma^{1/2}$. In particular $Y^A \cup X^{1/2}[\tau_A, \tau_{\kappa_j +1} ] $ does not contain any loop surrounding zero and in particular all the loops erased do not surround zero. As a consequence, $W(Y^1, 0) = W( Y^A \cup X^1[\tau_A, \tau_{\kappa_j +1} ], 0)$ and similarly for $X^2$. The construction was set up so that $W(Y^A \cup X^1) = W(Y^A \cup X_2) +2 \pi$ which concludes.
\end{proof}

\begin{corollary}\label{cor:contribution}
Let $j$ be an isolated scale, there exists a coupling of $\tilde X^1$ and $\tilde X^2$ two copies of the last piece of random walk at scale $j$ such that if $Y^1$ and $Y^2$ denote the full (forward) loop-erasures of the random walk, on a event $E$ of positive probability,
$$W(Y^2, 0) = W(Y^1, 0) + 2\pi.$$
\end{corollary}

Now we can finish the proof of the proposition.
\begin{proof}[Proof of \cref{prop:non-concentration}]
Recall that as above we sample the branch by first sampling the set $\cS$ then sampling everything outside of the pre-isolated scales and finally sampling everything outside of the isolated scales. Now for each isolated scale $i$, we consider a coupling as per \cref{cor:contribution}, independently at each scale. Since it is a coupling of two copies of the same law, we can choose independently with probability $1/2$ at each scale whether to use $X^1$ or $X^2$. Let $\ve_i $ be equal to $0$ if we choose $X^1$ and $1$ if we choose $X^2$. We start by sampling the $X^i$ and we let $\cI''$ be the set of scales where the event $E$ occur. Conditionally on the choices at all scales outside of $\cI''$, by \cref{cor:contribution}, the total winding is
\[
W(Y, 0) = W_0 + 2 \pi \sum_{i \in \cI''} \ve_i ,
\]
where $W_0$ is the winding if we only choose $X^1$. This is spread-out because it contains a binomial sum independent of everything else. 
\end{proof}

\bibliographystyle{alpha}

\bibliography{local_limit_robustness}
\end{document}